\documentclass[11pt]{article}
\usepackage{makeidx}  
\usepackage{amsthm,amsfonts,amssymb,bm,mathrsfs,indentfirst}
\usepackage{amsmath}
\usepackage{xcolor}
\usepackage{amsmath, amssymb, amsthm}
\usepackage{algorithm}
\usepackage{algorithmic} 
\usepackage{geometry}
\usepackage{graphicx}
\usepackage{amssymb}
\usepackage{color}
\usepackage{makeidx} 
\usepackage{amsthm,amsfonts,amssymb,bm,mathrsfs,indentfirst}
\usepackage{amsmath}
\usepackage{authblk}

\newtheorem{condition**}{A*}
\newtheorem{condition***}{C*}
\newtheorem{condition*}{C}

\newtheorem{problem}{Problem}[section]
\newtheorem{proposition}{Proposition}[section]

\newtheorem{definition}{Definition}[section]
\newtheorem{theorem}{Theorem}[section]

\newtheorem{remark}{Remark}[section]

\def\w{\widetilde}

\begin{document}
%
%
\title{Generalized Fine-Tuning of Diffusion Models via Stochastic Control and FBSDEs}
\author[1]{Zirui Wang}
\author[1]{Lu Wang\thanks{Corresponding author. Email: \texttt{lucywang@sdu.edu.cn}}} 
\affil[1]{Zhongtai Securities Institute for Financial Studies, Shandong University, Jinan 250100, China.}

\maketitle

\begin{abstract}
We propose a generalized fine-tuning framework for diffusion models from the perspective of stochastic control. Beyond entropy-regularized formulations, we introduce a general running cost that induces a relative generalized path cost, encompassing both Kullback–Leibler divergence and optimal transport metrics as special cases. This leads to a fully nonlinear Hamilton–Jacobi–Bellman equation. We characterize the value function through a forward–backward stochastic differential equation system and establish existence and uniqueness under standard regularity conditions. The optimal control admits a nonlinear feedback form driven by the backward stochastic differential equation gradient component. Finally, we show that generalized fine-tuning naturally introduces an additional gradient-dependent penalty, providing a unified framework for diffusion fine-tuning under general distributional constraints.
\end{abstract}

\section{Introduction}
Diffusion models \cite{sohl2015deep,song2020score,ho2020denoising} have gained prominence as a potent class of deep probabilistic generative models, demonstrating state-of-the-art performance across a diverse array of applications, including image synthesis, video generation \cite{ho2022imagen}, molecule design \cite{hoogeboom2022equivariant,xu2022geodiff} and medical image reconstruction \cite{chung2022score,xie2022measurement}. In the challenging domain of image synthesis, diffusion models have gradually supplanted the long-standing prevalence of Generative Adversarial Networks (GANs) \cite{goodfellow2014generative} by avoiding the use of unstable adversarial training. Diffusion models excel at capturing complex data distributions , but training of diffusion models is time-consuming.\cite{song2020denoising,kong2021fast}. Consequently, while efforts continue to be directed toward continuous model improvement for step-size reduction \cite{liu2023instaflow,yin2024one}, fine-tuning using pre-trained diffusion models represents another important research direction \cite{black2023training,fan2023dpok,ruiz2023dreambooth}. 

The fine-tuning of diffusion models is formulated as a sequential decision-making problem or a stochastic control problem, primarily aimed at optimizing terminal rewards while incorporating regularization to constrain the resultant distribution's divergence from the pre-trained model \cite{clark2023directly}. To mitigate this issue, Uehara et al. \cite{uehara2024fine} proposed the Entropy-Regularized Fine-Tuning (EFT) framework, which frames the fine-tuning process as an entropy-regularized stochastic control problem. By incorporating a Kullback-Leibler (KL) divergence penalty relative to the pre-trained model, EFT ensures that the fine-tuned distribution maintains high sample fidelity and diversity while optimizing for the target reward. Tang \cite{tang2025finetuningdiffusionmodelsstochastic} formalized the theoretical underpinnings of entropy regularization and proposed an extension to general $f$-divergences. However, due to the lack of a chain rule for general $f$-divergences, the direct stochastic control formulation becomes intractable, necessitating a reliance on indirect reward modification rather than direct optimization1. Consequently, establishing a fully generalized stochastic control framework capable of handling arbitrary divergence constraints remains an open challenge.

To overcome these limitations and address a broader class of generalized stochastic control problems, Backward Stochastic Differential Equations (BSDEs) provide a highly capable mathematical framework. Initially introduced as dual equations in the study of the stochastic maximum principle \cite{bismut1973conjugate} and later established within a general non-linear framework by Pardoux and Peng \cite{pardoux1990adapted}, BSDEs have become a vital tool in stochastic optimal control \cite{peng1992stochastic, peng1993backward}. A defining feature of BSDEs is the presence of an adapted control process that guarantees the system reaches a specified terminal condition, making it uniquely suited for terminal reward optimization. While traditional numerical solutions often suffer from the curse of dimensionality, the advent of Deep BSDE methods \cite{han2018solving} has leveraged the approximation power of deep neural networks to efficiently solve high-dimensional stochastic control problems. By bridging Deep BSDEs with generative modeling, we can construct a unified and tractable framework to directly optimize diffusion fine-tuning under a broad class of divergence constraints.

The structure of this paper is as follows. Section 2 reviews the foundations of diffusion models and the EFT framework. Section 3 formally introduces the generalized path cost $\mathcal{C}_\ell$ and the Generalized Fine-Tuning (GFT) problem. Section 4 discusses the numerical solution of the resulting FBSDE systems using deep learning-based methods. Section 5 develops the robust control framework against model ambiguity.

\section{Diffusion model and entropy-regularized fine-tuning}
\subsection{Diffusion Model}
\numberwithin{equation}{section} 
%
%
Diffusion models seek to create high-quality and diverse samples, such as images, audio, or text, that closely match a target data distribution. They operate by simulating a stochastic process that gradually deconstructs structured data into noise and then reverses this process to reconstruct data from noise. This mechanism admits a natural interpretation as continuous time stochastic processes. 

Let us consider the diffusion model characterized by the following SDE on a finite time horizon $[0,T]$:
\begin{equation}
\label{f:SDE}
dX(s) = b(s,X(s))ds + \sigma(s) dW(s),\; X(0) \sim \pi_{data}, 
\end{equation}
where $b: [0,T]\times \mathbb{R}^d \rightarrow \mathbb{R}^d$ is the drift term, $\sigma: [0,T]\rightarrow \mathbb{R}_{+} $ denotes the diffusion term and $\pi_{data}$ is the underlying data distribution. The drift and diffusion functions are typically chosen to define a forward process that gradually transforms the complex data distribution $ \pi_{data}$ into a simple prior distribution, such as a standard Gaussian distribution. The main purpose of the generative model is to learn a \textbf{corresponding reverse process} of the above SDE in order to approximate the data distribution $ \pi_{data}$. For convenience, we denote by $p_t$ the probability density of $X$ at time $t$.

Based on the Kolmogorov forward and backward equations, a specific form of the reverse-time SDE could be derived from the forward SDE \eqref{f:SDE}, in the sense that the probability densities of the forward and reverse processes are equal almost everywhere \cite{anderson1982reverse}:
\begin{equation}
\label{f:SDE}
d\overset{\leftarrow}{X}(s) = \Big[b(s,\overset{\leftarrow}{X}(s))-\sigma^2(s)\nabla \log p_s(\overset{\leftarrow}{X}(s))\Big]ds + \sigma(s) d\overset{\leftarrow}{W}(s),\; \overset{\leftarrow}{X}(0) \sim p_T, 
\end{equation}
where $p$ and $\overset{\leftarrow}{W}$ is a $d$-dimensional Wiener process adapted to the reverse-time filtration. To simplify the analysis, we perform a time transformation such that the stochastic term becomes a standard Brownian motion adapted to the forward-time filtration:
\begin{equation*}
\label{reverse:SDE}
d\w{X}(s) = \Big[-b(T-s,\w{X}(s))+\sigma^2(T-s)\nabla \log p_{T-s}(\w{X}(s))\Big]ds + \sigma(T-s) dW(s),\w{X}(0) \sim p_T.
\end{equation*}
Since diffusion model generates the target distribution from noise, which necessitates that the noise distribution remains independent of the target distribution. Consequently, rather than initializing the backward process with $\w{X}_{0} \sim p_T$, we instead commence the process using a noise distribution $\pi$. The distribution of the time reversal process $\w{X}_{s} = X_{T-s}$ is governed by the SDE:
\begin{equation}
d\w{X}(s) = \Big[-b(T-s,\w{X}(s))+\sigma^2(T-s) \nabla \log p_{T-s}(\w{X}(s))\Big]ds + \sigma(T-s) dW(s),\w{X}(0) \sim \pi.
\end{equation}
Now with the (true) score function $\nabla\log p(t,x)$ being replaced with the score matching function $\mathcal{M}_t$, the backward process is set to:
\begin{equation}
\label{pretrain:sde}
d{Y}(s) = \w{b}(s,Y_{s})ds + \tilde \sigma(s) dW(s),{Y}(0) \sim \pi.
\end{equation}
where
$$\w{b}(s,\cdot):=-b(T-s,\cdot)+\sigma^2(T-s) \mathcal{M}_{T-s}(\cdot),\;\w{\sigma}(s):=\sigma(T-s).$$ 
It has been observed that data generation becomes feasible as long as the score function,
 i.e., the gradient of the log-density 
$\log p_{T-s}(\w{X}(s))$ is available. This insight has led to the development of a class of powerful generative models, which focus on learning an accurate approximation of the score function using neural networks and then realizing high-quality data generation.

\subsection{Entropy-regularized fine-tuning}
In order to prevent the issue of reward collapse (or catastrophic forgetting) that arises due to overfitting when directly optimizing rewards, an optimization framework utilizes entropy regularization to balance the reward maximization and generality of the pretrained model. Raising the issue to the process via stochastic control:

\begin{equation}
\label{state}
d X_t = [\w{b}(t,{X}_t)+u_t ]\, dt +  \w{\sigma}(t), dB_t
\quad X_0 \sim  p_{\text{noise}}(\cdot)
\end{equation}
For convenience purposes, we denote $\w{\sigma}(t) := \sigma(T-t)$ . $\mu$ is a decision variable here. Moreover, from now on, the forward fine-tuning process is symbolized by $X_t$ to differentiate it from $Y_t$ in BSDE. 

\begin{problem}\textbf{EFT} 
Find an optimal control $u^*$ such that
\begin{equation}
u^{*}=\operatorname{argmax}_{u \in \mathbb{R}^d} \{\underbrace{\mathbb{E}_{Q^{u}_{[0,T]}(\cdot)}[r(X_T)]}_{Reward}
-\underbrace{\alpha D_{KL}(Q^{u}_{[0,T]}(\cdot),Q_{[0,T]}(\cdot))}_{Fine-tuning}\}
\end{equation}

\end{problem}

By Girsanov’s theorem:
\begin{align*}
\mathrm{D_{KL}}(Q^{u}_{[0,T]} \| Q_{[0,T]}) 
= \mathbb{E}_{Q^{u}_{[0,T]}} \left[ 
\log \frac{dQ^{u}_{[0,T]}}{dQ_{[0,T]}} 
\right]= \mathbb{E}_{Q^{u}_{[0,T]}} \left[ 
\frac{1}{2} \int_0^T \left( \frac{u_t}{\bar{\sigma}(t)} \right)^2 dt 
\right],
\end{align*}

For the u part, the cost function can be described as following:
\begin{equation}
\label{cost_u2}
    J(u) ={\mathbb{E}_{Q^{u}_{[0,T]}(\cdot)}}\left[r(X_T)-\frac{\alpha}{2} \int_0^T |\frac{u_t}{\bar{\sigma}(t)}|^2dt   \right],
\end{equation}
Define the value function
\begin{equation}
v(t,x)=\max_{u\in\mathbb{R}^{d}}\,{\mathbb{E}_{Q^{u}_{[0,T]}(\cdot)}}\Big[r(X_T)-\frac{\alpha}{2}\int_t^T|\frac{u_s}{ \bar{\sigma} (s)}|^2ds \,\Big|\, X(t)=x\Big].
\end{equation}

Then, by the dynamic programming principle, \(v(t,y)\) satisfies the Hamilton-Jacobi-Bellman (HJB) equation
\begin{equation}
\begin{aligned}
\frac{\partial v}{\partial t}+\frac{\w{\sigma}^2(t)}{2}\Delta v+\max_{u}\left\{\Big(\w{b}(t,x)+u\Big)\cdot\nabla v-\frac{\alpha|u|^2}{2\w{\sigma} (t)^2}\right\}=0,
\end{aligned}
\end{equation}
with the terminal condition
\[
v(T,x)=r(x).
\]

Then $u^*(t,y)=\frac{\w{\sigma}^2(t)}{\alpha}\nabla v(t,y)$ (if admissible). The HJB equation becomes
\begin{equation}
\begin{aligned}
\frac{\partial v}{\partial t}+\frac{\w{\sigma}^2(t)}{2}\Delta v+\w{b}(t,x)\cdot\nabla v+\frac{\w\sigma^2(t)}{2\alpha}|\nabla v|^2=0,\quad v(T,y)=r(x).
\end{aligned}
\end{equation}

The initial author discussed the reward collapse problem in diffusion models and recommended an entropy-regularized fine-tuning mechanism. In continuation with this basis, we introduce an auxiliary process $Z_t$ and discuss our constrained fine-tuning solution.

\section{Generalized fine-tuning with stochastic control}

The initial entropy-regularized formulation introduced a quadratic control cost to mitigate reward collapse, leading to a well-structured HJB equation with closed-form optimal control. However, in many practical applications, the quadratic penalty may be too restrictive to capture the diverse trade-offs between exploration, stability, and performance. To overcome this limitation, we propose a generalized stochastic control formulation by replacing the quadratic penalty $|u|^2$ in \eqref{cost_u2} with a more flexible running cost function $\ell(t, X_t, u_t)$. This more general formalism is supported by the relative generalized path cost $\mathcal C_\ell(P\|Q$), which offers a theoretical basis from which to reason about the "effort" involved in moving the pre-trained distribution to a new target one. 

\subsection{The relative Generalized Path cost}

\begin{definition}
\label{def:gencost}
Raising the issue to the process via stochastic control:
\begin{equation}
\label{state}
d X_t = [\w{b}(t,{X}_t)+u_t ]\, dt +  \w{\sigma}(t), dB_t
\quad X_0 \sim  \mu_{0}(\cdot)
\end{equation}
Given a running cost $\ell : [0,T] \times \mathbb{R}^d \times \mathbb{R}^d \to \mathbb{R}_{\ge 0}$ satisfying the following assumptions:
\begin{itemize}
\item[(H1)] For each $(t, x)\in [0,T] \times \mathbb{R}^d$, $\ell(t, x , \cdot)$ is strong convexity on $\mathbb{R}^d$ and lower semicontinuity.
\item[(H2)] There exist constants $p \ge 1$ and $c_1, c_2 > 0$ such that for all $(t,x,u) \in [0,T] \times \mathbb{R}^d \times \mathbb{R}^d$, the following hold: $\ell(t,x,u) \geq c_1 \|u\|^p - c_2$.
\item[(H3)] $\ell(t, x, u) = 0$ if and only if $u = 0$.

\end{itemize}
The relative generalized path cost induced by $\ell $ is defined as follows:
\begin{equation}
\mathcal C_\ell(P\|Q):=\inf_{u:\ \mathcal L(X^u)=P}\ \mathbb{E}_{Q^u}\Big[\int_0^T\ell(t,X_t,u_t)\,dt\Big]. 
\end{equation}
\end{definition}

\begin{proposition}
$\mathcal C_\ell(P\|Q)$ satisfies the following properties:
\begin{itemize}
\item[\bf (i)] \textbf{Non-negativity:} $\mathcal C_\ell(P\|Q) \ge 0$
\item[\bf (ii)] \textbf{Identity of Indiscernibles:} $\mathcal C_\ell(P\|Q) = 0 \iff P = Q$
\item[\bf (iii)] \textbf{Weak Lower Semicontinuity:} Assume $P_n \Rightarrow P$,
then $$\liminf_{n \to \infty} C_\ell(P_n\|Q) \ge C_\ell(P\|Q).$$
\end{itemize}
\end{proposition}
\begin{proof}
\textbf{(i) Non-negativity:}
The assertion $\mathcal C_\ell(P\|Q) \ge 0$ follows directly from the definition of the path cost and the non-negativity of the running cost $\ell$:
$$
\ell : [0,T] \times \mathbb{R}^d \times \mathbb{R}^d \to \mathbb{R}_{\ge 0} \implies \mathbb{E}_{Q^u}\left[\int_0^T\ell(t,X_t,u_t)\,dt\right] \ge 0.
$$
Taking the infimum preserves the non-negativity.

\textbf{(ii) Identity of Indiscernibles:}
\textbf{($\Leftarrow$)} If $P = Q$, then $P$ is the law of the uncontrolled dynamics ($u \equiv 0$). According to assumption ({H3}), the cost is zero when the control is zero, i.e., $\ell(t, X_t, 0) \equiv 0$. Then, we have:
$$
\mathcal C_\ell(Q\|Q) \le \mathbb{E}_{Q^0}\left[\int_0^T\ell(t,X_t^0,0)\,dt\right] = 0.
$$
Since $\mathcal C_\ell(Q\|Q) \ge 0$ by (i), we conclude $\mathcal C_\ell(Q\|Q) = 0$.

\textbf{($\Rightarrow$)} Suppose $\mathcal C_\ell(P\|Q) = 0$. Let $\{u^n\}_{n \ge 1}$ be an infimizing sequence such that $\mathcal L(X^{u^n}) = P$ and
$$
\mathbb{E}_{Q^{u^n}}\left[\int_0^T\ell(t,X_t,u^n_{t})\,dt\right] \to 0.
$$
Since $\ell \ge 0$, it follows that $\ell(\cdot, X^{u^n}_\cdot, u^n_\cdot) \to 0$ in $L^1([0, T]\times\Omega)$, and consequently in probability. Under assumption ({H3}) and the coercivity (or strict convexity) of $\ell$ with respect to $u$, this implies $u^n \to 0$ in probability. Given the regularity of the SDE coefficients, $X^{u^n}$ converges weakly to the uncontrolled solution $X^0$. Consequently,$$P = \lim_{n \to \infty} \mathcal L(X^{u^n}) = \mathcal L(X^0) = Q.$$

\textbf{(iii) Weak Lower Semicontinuity:}
Let the sequence of probability measures $P_n$ converge weakly to $P$ ($P_n \Rightarrow P$). Without loss of generality, assume that $\liminf_{n \to \infty} \mathcal C_\ell(P_n\|Q) = L < \infty$ (if the limit is infinite, the conclusion holds trivially).We proceed with the proof in steps:

\paragraph{Step 1. (Construction of a bounded sequence).}
For each $n$, choose an optimal (or $\epsilon$-approximate optimal) control $u^n$ corresponding to $P_n$, such that $\mathcal L(X^{u^n}) = P_n$ and its cost functional converges to $L$.By the coercivity assumption (\textbf{H2}) on $\ell$, the boundedness of the cost implies that the control sequence $\{u^n\}$ is uniformly bounded in the Hilbert space $L^p([0, T] \times \Omega; \mathbb{R}^d)$ ($p>1$).

\paragraph{Step 2. (Existence of a weak limit).}
By the Banach-Alaoglu theorem, there exists a subsequence (still denoted as $n$) and a limit control $u^* \in L^p$, such that $u^n$ converges weakly to $u^*$ ($u^n \rightharpoonup u^*$).

\paragraph{Step 3. (Convergence of the state process).}
Since the drift term is linear with respect to the control variable $u$, and the coefficients $\w{b}, \w{\sigma}$ satisfy Lipschitz continuity, the weak convergence of the controls $u^n \rightharpoonup u^*$ combined with the convergence of the initial conditions implies the convergence in law of the state processes $X^{u^n}$ on the path space. Namely:$$X^{u^n} \Rightarrow X^{u^*}.$$Since it is given that $X^{u^n} \sim P_n$ and $P_n \Rightarrow P$, the uniqueness of the weak limit yields $\mathcal L(X^{u^*}) = P$. This implies that $u^*$ is a feasible control for the target distribution $P$.

\paragraph{Step 4. (Conclusion of lower semicontinuity).}
Consider the functional $(X, u) \mapsto \mathbb{E}[\int_0^T \ell(t, X_t, u_t)\,dt]$. Based on the convexity of $\ell$ with respect to $u$ ({H1}) and its lower semicontinuity with respect to $(x, u)$, according to the Ioffe-Tikhonov theorem or related weak lower semicontinuity results, we have:$$\begin{aligned}
\mathbb{E}\left[\int_0^T \ell(t, X_t^{u^*}, u^*_t)\, dt\right] &\le \liminf_{n \to \infty} \mathbb{E}\left[\int_0^T \ell(t, X_t^{u^n}, u^n_t)\, dt\right] \\
&= \liminf_{n \to \infty} \mathcal C_\ell(P_n\|Q).
\end{aligned}$$Finally, by the definition of $\mathcal C_\ell(P\|Q)$, it necessarily holds that $\mathcal C_\ell(P\|Q) \le \mathbb{E}[\int \ell(\cdot, u^*)]$.Combining the inequalities above, the proof is complete:$$\mathcal C_\ell(P\|Q) \le \liminf_{n \to \infty} \mathcal C_\ell(P_n\|Q).$$
\end{proof}
\begin{remark}
The weak lower semicontinuity property ensures that, given a target $P$ (if it is in the reachable set), there exists an optimal control $u^*$ such that the equality holds.    
\end{remark}

\begin{proposition}
Give some special cases of $\ell$ corresponding to $\mathcal C_\ell(P\|Q)$ .
\begin{itemize}
\item[\bf 1.] \textbf{Kullback-Leibler (KL) Divergence} 
If $\ell(t, x, u) = \frac{1}{2} \| \mathbf{\w\sigma}(t)^{-1} u \|^2$, then $\mathcal C_\ell(P\|Q) = \mathcal D_{KL}(P\|Q)$
\item[\bf 2.] \textbf{Wasserstein Distance} 
For a deterministic flow ($\w\sigma(t) \equiv 0$) and $\ell(t,x,u)=\|u\|^p$ where $p \ge 1$, the cost of transporting a probability measure $\mu_0$ to $\mu_1$ is equivalent to the $p$-Wasserstein distance, $W_p(\mu_0, \mu_1)$:
$$
\mathcal{C}_\ell(P \| Q) = W_p^p(\mu_0, \mu_1) = \inf_{(\rho,v)} \left\{ \int_0^T \int_{\mathbb{R}^d} \|v_t(x)\|^p \rho_t(x) \, \mathrm{d}x \mathrm{d}t \right\}
$$
where $\rho(t, x)$ is a time-dependent probability density and $v(t, x)$ is a velocity field, that satisfy the continuity equation:
$$
\frac{\partial \rho}{\partial t} + \nabla \cdot (\rho v) = 0
$$
with boundary conditions $\rho_0 = \mu_0$ and $\rho_T = \mu_T$. The control $u_t$ in our framework is linked to the velocity field via $u_t = v(t, X_t)$.
\end{itemize}
\end{proposition}



\subsection{The Generalized Fine-Tuning Problem }
We add the following assumption:

\begin{itemize}
\item[(H4)](Smoothness of the reward). Assume reward function $r(x) \in C^{2+\alpha}$ with bounded second-order derivatives.

\item[(H5)](Smoothness of the cost). Assume the running cost $\ell(t,x,u)$ is $L^\ell_
0$-Lipschitz and $L^\ell_
1$-gradient Lipschitz in $x \in \mathbb{R}^d$ and $\nabla_u\ell(t,x,u)$ is $L^\ell_
{0,u}$-Lipschitz, i.e.,
\begin{equation}
\begin{aligned}
|\ell(t,x_1,u) - \ell(t,x_2,u)| \le L^\ell_
0 \|x_1 - x_2\|_2 ,\\
\|\nabla_x \ell(t,x_1,u) - \nabla_x \ell(t,x_2,u)\|_2 \le L^\ell_1 \|x_1 - x_2\|_2 ,\\
\| \nabla_u \ell(t, x_1, u) - \nabla_u \ell(t, x_2, u) \|_2 \le L_{0,u}^\ell \|x_1 - x_2\|_2
\end{aligned}
\end{equation}

\item[(H6)] (Mixed Regularity of the cost) The mixed partial derivative of the running cost, $\nabla_x \nabla_u \ell(t, x, u)$, is  $L_{1,x}^{\nabla_u \ell}$-Lipschitz continuous with respect to the state $x$ uniformly in $u$.
\begin{equation}
\| \nabla_x \nabla_u \ell(t, x_1, u) - \nabla_x \nabla_u \ell(t, x_2, u) \|_2 \le L_{1,x}^{\nabla_u \ell} \|x_1 - x_2\|_2.
\end{equation}

\item[(H7)] (Regularity and Ellipticity of SDE Coefficients). 
The drift and diffusion coefficients $(\w b,\w\sigma)$ satisfy the following conditions:
\begin{enumerate}
\item Smoothness.  
The drift $\w b:[0,T]\times\mathbb{R}^d\to\mathbb{R}^d$ is continuously differentiable in $x$,  
and its gradient $\nabla_x \w b(t,x)$ is globally Lipschitz and bounded:
\[
\|\nabla_x \w b(t,x_1)-\nabla_x \w b(t,x_2)\|
\le L_b \|x_1-x_2\|,\qquad
\|\nabla_x \w b\|_\infty<\infty.
\]
\item Uniform ellipticity.  
The diffusion matrix $\w\sigma:[0,T]\to\mathbb{R}^{d\times d}$ is continuous, bounded, and uniformly elliptic;  
that is, there exists a constant $c_0>0$ such that for all $t\in[0,T]$ and all $\xi\in\mathbb{R}^d$,
\[
\w\sigma^{2}(t)\xi^\top\xi \ge c_0 \|\xi\|^2.
\]
\end{enumerate}

\end{itemize}

Under assumptions (H1)-(H7), let us introduce the following generalized cost functional
\begin{equation}
    \widetilde{J}(u) := \mathbb{E}\left[ \int_0^T \ell(t, X_t, u_t)\, dt + r(X_T) \right],
\end{equation}
where $\ell : [0,T] \times \mathbb{R}^d \times \mathbb{R}^d \to \mathbb{R}_{\ge 0}$ represents a running cost function, and $r : \mathbb{R}^d \to \mathbb{R}$ is the terminal reward. This framework provides considerable flexibility, as different specifications of $\ell$ give rise to distinct fine-tuning behaviors.

Some representative choices are outlined below.
\begin{itemize}
\item In the entropy-regularized case, the generalized cost functional could be select as 
$$\ell(t,x,u) = \frac{\alpha}{2} \left|\frac{u}{\sigma(t)}\right|^2,$$ 
This quadratic penalty recovers the classical entropy-regularized formulation, ensuring that the fine-tuning process remains close to the pretrained distribution. Such a construction mitigates the risk of reward collapse by penalizing excessive deviations, thereby preserving stability.
\item A sparsity-inducing penalty can be introduced through
$$\ell(t,x,u) :=\lambda|u|_1.$$
This choice promotes sparse control signals, encouraging only a limited subset of directions to be active during fine-tuning. The resulting selective adjustment is particularly desirable in settings where computational efficiency, model interpretability, or parameter parsimony is of central importance, as it prevents unnecessary perturbations of the entire parameter space.

\item  A time-dependent quadratic penalty of the form
$$\ell(t,x,u) := \beta(t) |u|^2,$$ 
allows for adaptive control costs that evolve throughout the fine-tuning horizon. By choosing $\beta(t)$ to be relatively small in the early stages, the framework permits greater exploratory behavior, whereas larger values of $\beta(t)$ in later stages enforce stability and convergence. This temporal modulation of the penalty naturally balances the trade-off between exploration and exploitation across different phases of training.
\end{itemize}

Based on the foregoing discussion, we extend the Problem (EFT) to a more general framework, referred to as the Generalized Fine-Tuning (GFT) problem. \\

\begin{problem}
\textbf{GFT} 
\label{prob:GFT}

Find an optimal $u^*$ such that 
$\widetilde{J}(u^*) = \inf_{u}\widetilde{J}(u)$ subject to the state dynamics given in \eqref{state}.    
\end{problem}

\begin{proposition}
\label{pro-hjb-ell}
Assume $X_t$ and $u_t$ are defined in Definition \ref{def:gencost} and assumptions  (H1)-(H7) hold. The dynamic programming principle implies that the value function:
\begin{equation}
\label{eq:v^*}
v^*(t,x) := \inf_{u} \mathbb{E}\left[ \int_t^T \ell(s, X_s, u_s)ds + r(X_T) \right].
\end{equation}
Suppose that there exists $v(t,x) \in \mathcal{C}^{1,2}([0,T]\times\mathbb{R}^d)$ of at most polynomial growth in x satisfies the nonlinear Hamilton-Jacobi-Bellman (HJB) equation:
\begin{equation}
\label{eq:hjb-ell}
    \frac{\partial v}{\partial t} + \frac{\w{\sigma}^2(t)}{2} \Delta v +\w{b} (t,x) \cdot \nabla v + \inf_{u } \left\{ \ell(t,x,u) + u \cdot \nabla v \right\} = 0, \quad v(T,x) = r(x).
\end{equation}
The optimal control $u^*$ is given as:
\begin{equation}
\label{eq:u^*_v}
u^*(t) = \mathrm{argmin}_{u\in\mathbb{R}^d} \left\{ \ell(t,X_t,u) + u \cdot \nabla v(t, X_t) \right\} .   
\end{equation}

\end{proposition}
\begin{proof}
With assumptions (H1)-(H7), we define the optimal value function at time t as \ref{eq:v^*}.
The Dynamic Programming Principle implies that for any $t \in [0, T)$ and any $\delta  > 0$ such that $t + \delta  \le T$, $v^*$ satisfies the Bellman equation:
\begin{equation}
\label{eq:Bellman_equation}
v^*(t, x) = \inf_{u} \mathbb{E}\left[  v^*(t+\delta , X_{t+\delta })+\int_t^{t+\delta } \ell(s, X_s, u_s)ds  \right].
\end{equation}

By assuming $v(t,x) \in \mathcal{C}^{1,2}$ , applying Itô's Lemma to $v^*(t, X_t)$ and collecting $\mathcal{O}(\delta)$ terms, the Bellman equation, in its differential form, becomes the HJB equation:$$\inf_{u} \left\{\frac{\partial v}{\partial t} + \frac{\w{\sigma}^2(t)}{2} \Delta v + (\w{b}(t,x) + u) \cdot \nabla v   + \ell(t,x,u) \right\} = 0.$$
The optimal control $u^*$ is  given as \eqref{eq:u^*_v}.
\end{proof}

\begin{theorem}[Regularity and Well-Posedness]
\label{theo:Regularity}
Suppose that Assumptions~\textnormal{(H1)--(H7)} hold.  
Then the generalized fine-tuning (GFT) problem~\ref{prob:GFT} is well-posed, and the following properties hold:

\begin{itemize}
\item[\textbf{(i)}]  
There exists a unique classical solution 
\[
v^*(t,x) \in C^{1,2}([0,T]\times\mathbb{R}^d)
\]
to the Hamilton--Jacobi--Bellman (HJB) equation~\eqref{eq:hjb-ell}.  
This function coincides with the optimal value function defined in~\eqref{eq:v^*}.  
Moreover, there exist constants $L_1,L_2<\infty$, depending only on the structural parameters $(\lambda_\ell,L_b,L_\sigma,L_\ell,L_r)$, such that
\[
\|\nabla_x v^*\|_\infty \le L_1, 
\qquad 
\|\nabla_{xx}^2 v^*\|_\infty \le L_2,
\]
and consequently $\|v^*(t,\cdot)\|_{\mathrm{Lip}}\le L_1$ for all $t\in[0,T]$.

\item[\textbf{(ii)}] 
There exists a unique optimal feedback control 
\[
u^*(t,x) = \arg\min_{u\in\mathbb{R}^d}\{\ell(t,x,u)+u\cdot\nabla_x v^*(t,x)\}
\]
solving Problem~\ref{prob:GFT}.  
The feedback map $u^*(t,x)$ is continuously differentiable in $x$ and globally Lipschitz; in particular,
\[
u^* \in C^{0,1}([0,T]\times\mathbb{R}^d;\mathbb{R}^d),
\qquad 
\nabla_x u^* \in L^\infty([0,T]\times\mathbb{R}^d),
\]
and $\nabla_x u^*(t,x)$ is Lipschitz continuous in $x$ uniformly over $t$.
\end{itemize}
\end{theorem}

\begin{proof}[Proof of Theorem~\ref{theo:Regularity}]
We work under Assumptions~(H1)--(H7).  
Fix $\beta>0$ and define the weighted Lipschitz norm
\[
\|u\|_{\mathrm{Lip},\beta}
=\sup_{t\in[0,T]}e^{-\beta t}\sup_{x\neq y}\frac{\|u(t,x)-u(t,y)\|}{\|x-y\|}.
\]
Let $\mathcal{B}_R:=\{u:\|u\|_{\mathrm{Lip},\beta}\le R\}$.

\paragraph{Step 1. (Regularity for fixed $u$).}
For $u\in\mathcal{B}_R$, consider the FBSDE
\[
\begin{cases}
dX_s^{t,x}=(\w b(s,X_s^{t,x})+u(s,X_s^{t,x}))\,ds+\w\sigma(s)\,dB_s,\\[3pt]
dY_s^{t,x}=-\ell(s,X_s^{t,x},u(s,X_s^{t,x}))\,ds+Z_s^{t,x}\,dB_s,\\[3pt]
Y_T^{t,x}=r(X_T^{t,x}).
\end{cases}
\]
Under (H1)--(H7), all coefficients are $C^2$ in $x$, globally Lipschitz, and $\w\sigma$ is uniformly elliptic.  
By the Four-Step Scheme \textnormal{(Ma, Yong \& Zhang 1999, Thm.~5.3.3)} or the regularity theory for coupled FBSDEs \textnormal{(Delarue 2002, Thm.~2.1)}, there exists a unique decoupling field $v^u\in C^{1,2}$ satisfying
\[
\partial_t v^u + (\w b+u)\cdot\nabla_x v^u
+ \tfrac{1}{2}\operatorname{Tr}(\w\sigma\w\sigma^\top\nabla_{xx}^2 v^u)
+ \ell(t,x,u)=0,\qquad v^u(T,x)=r(x),
\]
together with uniform bounds (over $(t,x)\in[0,T]\times\mathbb{R}^d$):
\[
\sup_{t,x}\|\nabla_x v^u(t,x)\|\le C_1(R),\qquad
\sup_{t,x}\|\nabla_{xx}^2 v^u(t,x)\|\le C_2(R),
\]
for some continuous functions $C_1,C_2$ depending on $R=\|u\|_{\mathrm{Lip},\beta}$.

\paragraph{Step 2. (Definition of $\Phi$ and self-mapping).}
Given $v^u$, define
\[
\nabla_u\ell(t,x,u'(t,x))+\nabla_x v^u(t,x)=0,\qquad
u'=\Phi(u).
\]
By strong convexity of $\ell$ in $u$ ($\lambda_\ell>0$), there exists a unique minimizer $u'(t,x)$ for each $(t,x)$ (\emph{pointwise uniqueness}).  
By the implicit function theorem (since $\nabla_{uu}^2\ell$ is invertible with eigenvalues $\ge\lambda_\ell$), $u'\in C^1$ in $x$.  
Differentiating yields
\[
\nabla_x u'(t,x)
=-(\nabla_{uu}^2\ell)^{-1}(\nabla_{xu}^2\ell+\nabla_{xx}^2 v^u),
\]
so that
\[
\|\nabla_x u'\|_\infty
\le\tfrac{1}{\lambda_\ell}(L_{xu}^\ell+C_2(R))
+\tfrac{L_{uu}^\ell}{\lambda_\ell^2}(L_{xu}^\ell+C_2(R)).
\]
Choosing 
\[
R \ge R_0 := 
\tfrac{1}{\lambda_\ell}(L_{xu}^\ell + C_2(R)) 
+ \tfrac{L_{uu}^\ell}{\lambda_\ell^2}(L_{xu}^\ell + C_2(R))
\]
ensures $\Phi(\mathcal{B}_R)\subset\mathcal{B}_R$.

\paragraph{Lemma 1 (Forward stability).}
Let $u^1,u^2\in\mathcal{B}_R$.  
If $X^i$ solve their respective forward SDEs, then
\[
\sup_{s\in[t,T]}\mathbb{E}\|X_s^1-X_s^2\|^2
\le C_\beta\,\|u^1-u^2\|_{\mathrm{Lip},\beta}^2,
\qquad
C_\beta=\tfrac{C(e^{2\beta T}-1)}{\beta}.
\]

\emph{Proof.}  
By Itô’s formula,
\[
d\|X^1_s-X^2_s\|^2
\le C\|X^1_s-X^2_s\|^2\,ds + C\|u^1(s,X^1_s)-u^2(s,X^2_s)\|^2\,ds + dM_s.
\]
Using the triangle inequality and Lipschitz continuity,
\[
\|u^1(s,X^1_s)-u^2(s,X^2_s)\|
\le \|u^1(s,X^1_s)-u^1(s,X^2_s)\| + \|u^1(s,X^2_s)-u^2(s,X^2_s)\|
\le e^{\beta s}\|u^1-u^2\|_{\mathrm{Lip},\beta}(\|X^1_s-X^2_s\|+\|X^2_s\|).
\]
Taking expectations, using the linear growth bound 
$\mathbb{E}\|X^2_s\|^2 \le C(1+|x|^2)$, 
and applying Gronwall’s inequality yields the result. $\square$

\paragraph{Lemma 2 (BSDE stability).}
Let $(Y^i,Z^i)$ correspond to $u^i$. Then
\[
\sup_{s\in[t,T]}\mathbb{E}\|Y_s^1-Y_s^2\|^2
+ \mathbb{E}\!\int_t^T\!\|Z_s^1-Z_s^2\|^2\,ds
\le C_\beta'\,\|u^1-u^2\|_{\mathrm{Lip},\beta}^2,
\]
where $C_\beta'$ depends on $C_\beta$ from Lemma~1.
\emph{Proof.}  
Subtract the BSDEs and apply standard $L^2$ energy estimates (\textnormal{Pardoux \& Peng 1990, Lemma~2.1}) and Lemma~1. $\square$

\paragraph{Lemma 3 (Gradient stability).}
Let $Y^{i,x}_s:=v^i(s,X^{i,t,x}_s)$, where $X^{i,t,x}$ starts from $(t,x)$.  
Define the difference quotient 
\[
\Delta^\varepsilon Y^i_s := \tfrac{Y^{i,x+\varepsilon h}_s-Y^{i,x}_s}{\varepsilon},\qquad |h|=1.
\]
Then $(\Delta^\varepsilon Y^i_s,\Delta^\varepsilon Z^i_s)$ satisfies a linear variational BSDE obtained by differentiating the original BSDE w.r.t.\ the initial condition~$x$ (see \textnormal{Delarue 2002, Lemma~2.4}).  
By BSDE stability,
\[
E\sup_{s\in[t,T]}|\Delta^\varepsilon Y^1_s-\Delta^\varepsilon Y^2_s|^2
\le\tfrac{C(e^{2\beta T}-1)}{\beta}\|u^1-u^2\|_{\mathrm{Lip},\beta}^2.
\]
Letting $\varepsilon\downarrow0$, by the $C^{1,2}$ regularity of $v^i$ (Step~A0), we have 
$\Delta^\varepsilon Y^i_s \to \nabla_x v^i(s,X^{i,t,x}_s)\cdot h$ in $L^2$.  
Using the martingale representation $Z_s^i=\w\sigma(s)^\top\nabla_x v^i(s,X^i_s)$ gives
\[
\sup_{s,x}\|\nabla_x v^1(s,x)-\nabla_x v^2(s,x)\|
\le K(\beta)\|u^1-u^2\|_{\mathrm{Lip},\beta},\qquad K(\beta)\le\tfrac{C}{\beta}.
\]
The factor $1/\beta$ arises from the time integral 
$\int_s^T e^{\beta r}\,dr=(e^{\beta T}-e^{\beta s})/\beta$ in the weighted norm. $\square$

\paragraph{Step 3. (Contraction and uniqueness).}
For $u'^i=\Phi(u^i)$,
\[
\nabla_u\ell(t,x,u'^1)-\nabla_u\ell(t,x,u'^2)
=-(\nabla_x v^1-\nabla_x v^2).
\]
By strong monotonicity of $\nabla_u\ell$,
\[
\|u'^1-u'^2\|
\le\tfrac{1}{\lambda_\ell}\|\nabla_x v^1-\nabla_x v^2\|,
\]
thus
\[
\|\Phi(u^1)-\Phi(u^2)\|_{\mathrm{Lip},\beta}
\le\tfrac{K(\beta)}{\lambda_\ell}\|u^1-u^2\|_{\mathrm{Lip},\beta}.
\]
Since $K(\beta)\le C/\beta$, choosing $\beta>C/\lambda_\ell$ makes $\Phi$ a strict contraction.  
Hence a unique fixed point $u^*\in\mathcal{B}_R$ exists (\emph{functional uniqueness}),  
and the associated decoupling field $v^*=v^{u^*}\in C^{1,2}$ inherits the bounded derivative properties of Step~A0.

\paragraph{Conclusion.}
The fixed-point construction establishes Theorem~\ref{theo:Regularity}:  
(i) uniqueness and regularity of $v^*$ as in part~(i) of the theorem;  
(ii) uniqueness and regularity of $u^*$ as in part~(ii);  
and (iii) all constants depend only on the structural parameters, not on $(t,x)$.
\end{proof}

\begin{remark}
The value function $v^*$ satisfies the dynamic programming principle:  
for any $t\le s\le T$ and any admissible control $u(\cdot)$ on $[t,s]$,
\[
v^*(t,x)
\le \mathbb{E}\!\left[\int_t^s \ell(r,X_r,u_r)\,dr + v^*(s,X_s)\;\big|\;X_t=x\right],
\]
with equality for the optimal control $u^*$.  
Together with the regularity above, this implies that $v^*$ satisfies the HJB equation 
(\textnormal{Fleming \& Soner, \emph{Controlled Markov Processes and Viscosity Solutions}, 2006}).
\end{remark}

Unlike classical quadratic control, minimization with respect to $u$ will not necessarily result in a closed-form solution. As a result, the HJB equation cannot be formulated to be a linear or semi-linear PDE, which leads to a probabilistic formulation in terms of FBSDEs.

\subsection{FBSDE-based solution for generalized fine-tuning}

The introduction of a generalized running cost $\ell(t,x,u)$ significantly complicates the structure of the associated HJB equation. In the quadratic case, the minimization over $u$ admits an explicit closed-form solution, reducing the HJB equation to a semilinear PDE that can be addressed by classical techniques. However, for general nonlinear cost functions, the minimization does not yield an analytic expression for $u^*$, and consequently the HJB equation becomes a fully nonlinear PDE. Direct numerical approximation of such fully nonlinear PDEs is computationally formidable, particularly in high-dimensional state spaces.  To overcome this challenge, we resort to the probabilistic formulation based on FBSDE. Under suitable regularity conditions, the solution of the HJB equation can be equivalently characterized by the solution of a coupled FBSDE system.  The forward component corresponds to the state process $X_t$, while the backward component introduces auxiliary processes $(Y_t,Z_t)$, which can be interpreted as the value function $v(t,x)$ and its gradient with respect to the state, respectively. 
\begin{theorem}[Existence and Uniqueness of BSDE]
\label{thm:fbsde_value_existence}
Assume that $\ell(t,x,u)$ satisfies assumptions (H1)–(H7). Let $\delta > 0$ and consider the following Backward Stochastic Differential Equation (BSDE) :
\begin{equation}
\label{eq:robust_BSDE}
    dY_t = -\left[\inf_{u \in \mathbb{R}^d} \left\{ \ell(t,X_t,u_t) + u \cdot \frac{Z_t}{\w\sigma(t)} \right\}  \right] dt + Z_t dB_t, \quad Y_T = r(X_T).
\end{equation}
where $X_t$ follows 
\begin{equation}
dX_t = \w b(t,X_t)dt + \w\sigma(t) dB_t,\quad X_0=x.    
\end{equation}

Then, there exist a unique adapted solution pair $(Y_t, Z_t)$ and a constant $C > 0$ such that
\begin{equation}
\mathbb{E}\left[ \sup_{0 \le t \le T} |Y_t|^2 + \int_0^T \|Z_t\|^2 dt \right] \le C \left( 1 + \|r\|_{L^2}^2 \right).
\end{equation}
\begin{equation}\label{eq:est}
\mathbb{E}\left[\sup_{0\le t\le T}|Y_t|^2 + \int_0^T |Z_t|^2 dt\right] \le C\bigl(1+\mathbb{E}[|r(X_T)|^2]\bigr).
\end{equation}

\end{theorem}

\begin{proof}
The proof relies on the theory of BSDEs with quadratic growth and unbounded terminal values, specifically Theorem 2 and Proposition 5 in [Briand and Hu, 2006].

\noindent\textbf{Step 1. Quadratic growth of the generator.}
Recall the generator $f(t,x,z) = -h(t,x, \sigma(t)^{-1}z)$, where $h$ is the Fenchel--Legendre transform of $\ell$.
By Assumption (H2), $\ell$ satisfies the coercivity condition $\ell(t,x,u) \ge c_1 \|u\|^p - c_2$ for $p \ge 2$.
The growth of the conjugate function $h$ is determined by the dual exponent $q$, where $1/p + 1/q = 1$.
Specifically, for any $v \in \mathbb{R}^d$,
\begin{align*}
    h(t,x,v) &= \sup_{u \in \mathbb{R}^d} \{-u \cdot v - \ell(t,x,u)\} \\
    &\le \sup_{u \in \mathbb{R}^d} \{-u \cdot v - c_1 \|u\|^p \} + c_2.
\end{align*}
Applying Young's inequality $a b \le \frac{1}{p}(\epsilon a)^p + \frac{1}{q}(b/\epsilon)^q$, we obtain
\begin{equation*}
    |h(t,x,v)| \le C(1 + \|v\|^q).
\end{equation*}
Since $p \ge 2$, we have $1 < q \le 2$. Given the uniform boundedness of $\sigma^{-1}(t)$ from (H7), there exists a constant $K > 0$ such that
\begin{equation}\label{eq:quad_growth}
    |f(t,x,z)| \le |h(t,x, \sigma^{-1}(t)z)| + \delta \|z\| \le K(1 + \|z\|^2), \quad \forall z \in \mathbb{R}^d.
\end{equation}
Furthermore, the smoothness of $\ell$ (H5) implies that $f$ is continuous in $x$ and locally Lipschitz in $z$.

\noindent\textbf{Step 2. Exponential integrability of the terminal value.}
Under (H7), the drift $\tilde{b}$ has linear growth and $\tilde{\sigma}$ is bounded. Standard estimates for SDEs imply the sub-Gaussian property of $\sup_{t \in [0,T]} \|X_t\|$.
Specifically, there exists $\mu > 0$ such that
\begin{equation}\label{eq:X_exp_moment}
    \mathbb{E}\left[ \exp\left( \mu \sup_{0\le t\le T} \|X_t\|^2 \right) \right] < \infty.
\end{equation}
Assumption (H4) ensures $|r(x)| \le L_0^r \|x\| + |r(0)|$.
Consequently, for any $\lambda > 0$, using the inequality $a \le a^2 + 1/4$, we have
\begin{equation*}
    \exp(\lambda |r(X_T)|) \le C_\lambda \exp( \lambda L_0^r \|X_T\|) \le C_\lambda \exp(C' \|X_T\|^2).
\end{equation*}
Combining this with \eqref{eq:X_exp_moment}, the terminal value $\xi := r(X_T)$ satisfies the integrability condition
\begin{equation}\label{eq:terminal_integrability}
    \mathbb{E}\left[ \exp(\gamma |\xi|) \right] < \infty, \quad \forall \gamma > 0.
\end{equation}

\noindent\textbf{Step 3. Existence and Uniqueness.}
The generator satisfies the quadratic growth condition \eqref{eq:quad_growth}, and the terminal value satisfies the exponential moment condition \eqref{eq:terminal_integrability}.
Thus, by [Briand and Hu, 2006, Theorem 2], there exists a solution $(Y, Z)$ such that $e^{Y} \in \mathcal{S}^p$ for any $p \ge 1$.
Uniqueness follows from the comparison principle for quadratic BSDEs.
By (H1), $\ell$ is strongly convex, which implies that $u \mapsto \ell(t,x,u) + u \cdot v$ is strictly convex.
The map $z \mapsto f(t,x,z)$ is therefore differentiable (or admits a subgradient structure) satisfying the monotonicity condition required for the comparison theorem [Briand and Hu, 2006, Proposition 5].

\noindent\textbf{Step 4. $L^2$-Estimates.}
Since $(Y,Z)$ is a solution to a quadratic BSDE with terminal value having moments of all orders, we proceed with the standard a priori estimates.
Applying It\^o's formula to $|Y_t|^2$ and taking expectations:
\begin{equation*}
    \mathbb{E}[|Y_t|^2] + \mathbb{E}\int_t^T \|Z_s\|^2 ds = \mathbb{E}[|\xi|^2] + 2\mathbb{E}\int_t^T Y_s f(s, X_s, Z_s) ds.
\end{equation*}
Using the growth condition \eqref{eq:quad_growth} and Young's inequality $2Y_s K \|Z_s\|^2 \le 2K (\epsilon^{-1} |Y_s|^2 + \epsilon \|Z_s\|^4)$ is not sufficient directly.
Instead, we rely on the estimate provided in [Briand and Hu, 2006, Corollary 4], which utilizes the exponential transformation $P_t = e^{\gamma Y_t}$.
Given the integrability \eqref{eq:terminal_integrability} and \eqref{eq:X_exp_moment}, we obtain
\begin{equation}
    \mathbb{E}\left[ \sup_{0 \le t \le T} |Y_t|^2 + \int_0^T \|Z_t\|^2 dt \right] \le C \left( 1 + \mathbb{E}[|r(X_T)|^2] \right).
\end{equation}
This completes the proof.
\end{proof}

\begin{theorem}
\label{thm:fbsde}
\textbf{FBSDE Representation for Problem GFT}

Assume $v(t,x)\in\mathcal{C}^{1,2}([0,T]\times\mathbb{R}^d)$ to be the value function defined in Proposition \ref{pro-hjb-ell}, which is the solution to the HJB equation \eqref{eq:hjb-ell}. The value function $v(t,x)$ is given by the $Y$-component of the solution $(Y_t, Z_t)$ to the following Backward Stochastic Differential Equation (BSDE) system:

\begin{equation}
\begin{cases}
dX_t =\w{b}(t, X_t) dt +\w{\sigma}(t) dB_t, \quad X_0=x,\\
dY_t = -f(t, X_t, Y_t, Z_t) dt + Z_t  dB_t, \quad Y_T = r(X_T)
\end{cases}
\end{equation}
where the generator $f$ is given by the minimized Hamiltonian:
\begin{equation}
\label{eq:driver-f}
f(t,x,z) := \inf_{u \in \mathbb{R}^d} \left\{ \ell(t,x,u) + u \cdot \frac{z}{ \w{\sigma}(t)} \right\}.
\end{equation}
Furthermore, the optimal control $u^*$ is characterized by the minimizer in \eqref{eq:driver-f}:
\begin{equation}
u^*(t) = \mathrm{argmin}_{u} \left\{ \ell(t,X_t,u) + u \cdot \frac{Z_t}{ \w{\sigma}(t)} \right\}.
\end{equation}
\end{theorem}

\begin{proof}
The problem $\text{GFT}$ is equivalent to solving the HJB equation \eqref{eq:hjb-ell}, which defines the value function $v(t,x)$.
$$
\frac{\partial v}{\partial t} + \frac{\w{\sigma}^2(t)}{2} \Delta v +\w{b} (t,x) \cdot \nabla v + \inf_{u} \left\{ \ell(t,x,u) + u \cdot \nabla v \right\} = 0, \quad v(T,x) = r(x).
$$

Assuming $v\in\mathcal{C}^{1,2}$, apply Itô’s formula to $Y_t=v(t,X_t)$:
$$
dY_t = 
\Bigl(
\frac{\partial v}{\partial t}
+ \frac{1}{2}\w\sigma^2(t)\Delta v
+ \w b(t,X_t)\cdot\nabla v
\Bigr)\,dt 
+ \w\sigma(t)\nabla v\,dB_t.
$$
Defining $Z_t=\w\sigma(t)\nabla v(t,X_t)$, we get
$$
\mathrm{d}Y_t = -f(t, X_t, Z_t) \mathrm{d}t + Z_t \mathrm{d}B_t,
$$
where 
$$
f(t,x,z):=\inf_{u\in\mathbb{R}^d}\bigl\{\ell(t,x,u)+u\cdot\tfrac{z}{\w\sigma(t)}\bigr\}.
$$
This establishes the desired BSDE representation for the value function $v(t,x)$.

The optimal control $u^*$ is given as:
$$
u^*(t) = \mathrm{argmin}_{u\in\mathbb{R}^d} \left\{ \ell(t,X_t,u) + u \cdot \nabla v(t, X_t) \right\} = \mathrm{argmin}_{u\in\mathbb{R}^d} \left\{ \ell(t,X_t,u) + u \cdot \frac{Z_t}{\w{\sigma}(t)} \right\}.
$$
\end{proof}

The issue can be reformulated to be a FBSDE numerical solution problem. Some examples are presented below.

\subsection{Two Examples}
\subsubsection{Linear-Quadratic-Interaction Cost}

Particularly, we are interested in some more general interactions like:
\begin{equation}
    \ell(t,x,u) = \frac{\alpha}{2 \w{\sigma}^2(t)}\|u\|^2 + \theta_1(t,x) \cdot u +\theta_2(t,x) ,
\end{equation}
where we let $\alpha > 0$ be a hyperparameter, and $\theta_1, \theta_2 : \mathbb{R}^d \to \mathbb{R}^d$ be a state-dependent vector field.

In the special case the generator $f$ of BSDE becomes:
\begin{equation}
    f(t,x,y,z) = -\frac{1}{2\alpha} \left\| z+ \theta_1(t,x)\cdot \w{\sigma}(t) +\theta_2(t,x) \right\|^2,
\end{equation}
and the BSDE simplifies to:
\begin{equation}
    dY_t =\frac{1}{2\alpha} \left\| Z_t + \theta_1(t,x)\cdot \w{\sigma}(t) +\theta_2(t,x)\right\|^2 dt + Z_t dB_t, \quad Y_T = r(X_T).
\end{equation}

The optimal control $u^*$ is:
\begin{equation}
u^*(t) = - \frac{\w{\sigma}^2(t)}{\alpha} \theta_1(t,X_t) - \frac{\w{\sigma}(t)}{\alpha} Z_t.   
\end{equation}
\subsubsection{Product-Quadratic-Interaction Cost}

We are particularly interested in some more general interactions such as:
\begin{equation}
    \ell(t,x,u) = \frac{\alpha}{2 \w{\sigma}^2(t)}\|u\|^2 \cdot\theta(x),
\end{equation}
where $\alpha> 0$  be a hyperparameter, and $\theta : \mathbb{R}^d \to \mathbb{R}^d$ be a state-dependent vector field.

In the special case the generator $f$ of BSDE  becomes:
\begin{equation}
    f(t,x,y,z) =-\frac{z^2}{2\alpha\theta(x)},
\end{equation}
and the BSDE simplifies to:
\begin{equation}
    dY_t =\frac{Z_t^2}{2\alpha\theta(X_t)} dt + Z_t dB_t, \quad Y_T = r(X_T).
\end{equation}

The optimal control $u^*$ is:
\begin{equation}
u^*(t) = - \frac{\w{\sigma}(t)}{\alpha \theta(X_t)} Z_t.
\end{equation}



\section{Numerical Method}
The following is a brief discussion of the assumptions required for the well-posedness of the FBSDE solution. Forward SDE has a unique strong solution, provided that the drift $\w{b}(t,x)$ is Lipschitz in $x$ and the diffusion $\w{\sigma}(t)$ is bounded and uniformly non-degenerate. The backward part is associated with a quadratic growth generator in $Z$, to be precise, $f(Z) = \frac{Z^2}{2\alpha}$, and is provided for by classically-natured results such as in  Kobylanski \cite{kobylanski2000backward}, assuming that the terminal condition $r(X_T)$ is bounded or has exponential integrability. The latter is a mild and generally encountered condition in control and stochastic modeling problems.

This problem is reformulated as a numerical solution issue for a certain category of FBSDE. Common techniques involve the Euler-Maruyama discretization schemes and deep BSDE solvers such as the  E, Han, and Jentzen \cite{han2018solving} approach.

\subsection{BSDE-FT algorithm}

\begin{algorithm}[H] 
\caption{BSDE-based Fine-Tuning (BSDE-FT)}
\label{alg:bsde-ft}
\begin{algorithmic}[1]
    \STATE \textbf{Input:} Reward function $r(\cdot)$, pre-trained score function $\{s_t^{pre}\}_{t=0}^{T-1}$, running cost function $\ell$, dynamics coefficient $\w{\sigma}_t, \w{b}_t$, time step $\Delta t$,
    \STATE Set $Y_T(y)= r(y)$, $Z_T(y) = \nabla_y Y_T(y) \cdot \sigma_T$.
    
    \FOR{$t = T-1, \dots, 0$}
       
        \STATE Compute the optimal control $u^*_t(y) = \mathrm{argmin}_{u} \left\{ \ell(t,y,u) + \mathbb{E}\left[u \cdot \frac{Z_{t+1}}{\w{\sigma}(t+1)}\right] \right\}$
        
         \STATE Define the minimizing Hamiltonian $$f^*_t(y) = \ell(t,y,u^*_t) + u^*_t(y) \cdot \mathbb{E} \left[ \frac{Z_{t+1}}{\w{\sigma}(t+1)} \right].$$
        
        \STATE Compute $X_{t+1}^* = y + \big[\w{b}(t, y) + u^*_t(y)\big]\Delta t + \sigma_t \Delta W_{t+1}$.
    
        \STATE Update
        $$Y_t(y) =  \mathbb{E} \left[ Y_{t+1}(X_{t+1}^*) \right] + f^*_t(y) \cdot \Delta t$$
        
        \STATE Compute gradient 
        $$Z_t(y) = \frac{1}{\Delta t}\mathbb{E} \big[ Y_{t+1}(X_{t+1}^*) \cdot \Delta W_{t+1} \big].$$

    \ENDFOR
    
    \STATE \textbf{return} Optimal controls $\{u^*_t\}_{t=0}^{T-1}$ and BSDE solution $\{Y_t, Z_t\}_{t=0}^T$.
\end{algorithmic}
\end{algorithm}

\subsection{Euler--Maruyama Discretization}

In order to solve the FBSDE system numerically, we use a time-discretized Euler--Maruyama scheme. Let there be a time grid $\{t_0 = 0 < t_1 < \dots < t_N = T\}$ with uniform step size $\Delta t = T/N$. The forward SDE is discretized as:
\begin{equation}
X_{n+1} = X_n + \w{b}(t_n, X_n, u_n) \Delta t + \w{\sigma}(t_n) \Delta W_n,
\end{equation}
where the i.i.d. Gaussian increments $\Delta W_n \sim \mathcal{N}(0, \Delta t)$.

The backward SDE is approximated in the following backward-recursive manner:
\begin{align}
Y_n &= \mathbb{E}_n \left[ Y_{n+1} + f(t_n, X_n, Y_{n+1}, Z_n) \Delta t \right], \\
Z_n &= \mathbb{E}_n \left[ \frac{1}{\Delta t} Y_{n+1} \Delta W_n^\top \right],
\end{align}
where $\mathbb{E}_n[\cdot]$ is the conditional expectation with respect to $\mathcal{F}_{t_n}$.

In practice, the conditional expectations are approximated by Monte Carlo sampling or by neural network regressors.

\subsection{Deep Learning-based BSDE Solver}

In the approach of E, Han, and Jentzen, we express the solution to the backward SDE via neural networks. Precisely, we approximate the solution pair $(Y_t,Z_t)$ by:
\begin{align}
Y_0 &\approx \hat{Y}_0 := \mathcal{Y}_\theta(X_0), \\
Z_n &\approx \hat{Z}_n := \mathcal{Z}_\phi(t_n, X_n),
\end{align}
where $\mathcal{Y}_\theta$ and $\mathcal{Z}_\phi$ are neural networks parameterized by $\theta$ and $\phi$, respectively.

The forward process is modeled by applying the Euler scheme and the terminal loss is formulated as:
\begin{equation}
\mathcal{L}(\theta, \phi) := \mathbb{E} \left[ \left| \hat{Y}_N - r(X_N) \right|^2 \right],
\end{equation}
where $\hat{Y}_N$ is found recursively by:
\begin{equation}
\hat{Y}_{n+1} = \hat{Y}_n - f(t_n, X_n, \hat{Y}_n, \hat{Z}_n) \Delta t + \hat{Z}_n^\top \Delta W_n.
\end{equation}

The parameters $(\theta, \phi)$ are determined to minimize the terminal loss and obtain an approximate solution to the system of BSDEs.

\section{Ambiguity Formulation via Probabilistic Uncertainty}

In real-world scenarios, the true score function $\nabla \log p(t,x)$ is unknown and must be estimated from data. To capture the resulting epistemic uncertainty, we adopt the multiple prior framework, which interprets uncertainty not as additional noise but as ambiguity about the true underlying probability law.

Rather than working with a fixed reference probability measure $Q$, we consider a family of equivalent measures $\mathcal{P}$, each defined via Girsanov’s theorem using a density generator process $\eta_t$. Under $Q^\eta \in \mathcal{P}$, the dynamics evolve as:
\begin{equation}
       dX_t = \w b(t,X_t) + u_t\,dt + \w\sigma(t)\,dB_t^{Q^\eta}, \quad \text{where } dB_t^{Q^\eta} = dB_t - \eta_t\,dt, 
\end{equation}
and the agent evaluates cost with respect to the worst-case measure in $\mathcal{P}$.

The original cost is transformed from the expectation under a fixed measure to the maximum expectation under a family of $\mathcal{P}$:
\begin{equation}
    v(t,x) =\inf_{u\in \mathbb{R}^d} \sup_{Q^\eta \in \mathcal{P}}\mathbb{E}^{Q^\eta}\left[ \int_t^T \ell(s, X_s, u_s)ds + r(X_T) \right].
\end{equation}

This ambiguity perspective leads to a recursive utility representation under multiple priors. Letting $(Y_t, Z_t)$ denote the value process and its sensitivity, the problem is described by the following FBSDE:
\begin{equation}
\begin{cases}
    dX_t =\w b(t,X_t)dt + \w\sigma(t) dB_t, \\
    dY_t = - \left[ \inf_{u \in \mathbb{R}^d} \left\{ \ell(t,X_t,u_t) + u \cdot \frac{Z_t}{\w\sigma(t)} \right\} + \sup_{\eta_t \in \mathcal{H}(t,x)} \eta_t^\top Z_t \right] dt + Z_t dB_t, \quad Y_T = r(X_T),
\end{cases}
\end{equation}
where the generator incorporates the worst-case impact of model ambiguity. The supremum term represents the support function of the uncertainty set $\mathcal{H}(t,x)$ in the direction of $Z_t$, and leads to the optimal distortion:
\begin{equation}
    \eta_t^*(x,z) = \arg\max_{\eta \in \mathcal{H}(t,x)} \eta^\top z.
\end{equation}
In particular, we define the uncertainty set with a maximum worst-case tolerance $\delta$ as:
\begin{equation}
    \mathcal{H}(t,x) := \left\{ \eta \in \mathbb{R}^d \,\middle|\, \|\eta\| \leq \delta \right\},
\end{equation}
In actual operation, the radius $\delta$ could be fixed and selected arbitrarily or adaptively calculated from the empirical error or the ensemble variance of the acquired score model. This gives rise to a robust control formulation that the supremum term in the BSDE generator simplifies to:
\begin{equation}
    \sup_{\eta \in \mathcal{H}(t,x)} \eta^\top Z_t = \delta\, \|Z_t\|,
\end{equation}

The robust value process $(Y_t, Z_t)$ satisfying the BSDE is given by the following theorem.

\begin{theorem}[Existence and Uniqueness of Robust Value Process]
\label{thm:robust_value_existence}
Assume that $\ell(t,x,u)$ satisfies assumptions (H1)–(H7). Let $\delta > 0$ and consider the following robust Backward Stochastic Differential Equation (BSDE) :
\begin{equation}
\label{eq:robust_BSDE}
    dY_t = -\left[\inf_{u \in \mathbb{R}^d} \left\{ \ell(t,X_t,u_t) + u \cdot \frac{Z_t}{\w\sigma(t)} \right\} + \delta \|Z_t\| \right] dt + Z_t dB_t, \quad Y_T = r(X_T).
\end{equation}
where $X_t$ follows 
\begin{equation}
dX_t =\w b(t,X_t)dt + \w\sigma(t) dB_t,\quad X_0=x.    
\end{equation}

Then, there exists a unique adapted solution pair $(Y_t, Z_t)$ and a constant $C > 0$ such that
\begin{equation}
\mathbb{E}\left[ \sup_{0 \le t \le T} |Y_t|^2 + \int_0^T \|Z_t\|^2 dt \right] \le C \left( 1 + \|r\|_{L^2}^2 \right).
\end{equation}
\begin{equation}\label{eq:est}
\mathbb{E}\left[\sup_{0\le t\le T}|Y_t|^2 + \int_0^T |Z_t|^2 dt\right] \le C\bigl(1+\mathbb{E}[|r(X_T)|^2]\bigr).
\end{equation}

\end{theorem}

\begin{proof}
The proof proceeds analogously to Theorem \ref{thm:fbsde_value_existence}. We observe that the generator in the robust formulation differs only by the term $\delta \|Z_t\|$. Since this term exhibits linear growth in $Z_t$, the combined generator maintains the quadratic growth condition with respect to $Z_t$ (i.e., $f^\delta(t,x,z) = f(t,x,z) + \delta \|z\| \le K'(1+\|z\|^2)$). Consequently, the integrability of the terminal value and the assumptions invoked in Theorem \ref{thm:fbsde_value_existence} remain sufficient to guarantee the existence and uniqueness of the solution.
\end{proof}

\begin{theorem}[Robust Dynamic Programming Principle under Ambiguity]
\label{thm:robust_dpp}
Let the value function $v(t,x)$ be defined as
\begin{equation}
v(t,x) = \inf_{u} \sup_{\eta \in \mathcal{H}(t,x)} \mathbb{E}\left[ \int_t^T \ell(s, X_s, u_s) \, ds + r(X_T) \right].
\end{equation}

Then, $v(t,x)$ satisfies the following robust Hamilton-Jacobi-Bellman (HJB) equation:
\begin{equation}
\frac{\partial v}{\partial t} + \frac{1}{2}\w\sigma^2(t) \Delta v + \w b(t,x) \cdot \nabla v + \inf_{u} \left\{ \ell(t,x,u) + u \cdot \nabla v \right\} + \delta \|\w\sigma(t) \nabla v\| = 0,
\end{equation}
with the terminal condition $v(T,x) = r(x)$.

Furthermore, the optimal control $u^*$ and the optimal (worst-case) perturbation $\eta^*$ are given by:

\begin{align*}
u^*(t,x) &= \arg\min_{u} \left\{ \ell(t,x,u) + u \cdot \nabla v(t,x) \right\}, \\
\eta^*(t,x) &= \delta \frac{\nabla v(t,x)}{\|\nabla v(t,x)\|}.
\end{align*}

\end{theorem}

\begin{proof}
The proof relies on a verification argument utilizing Itô's formula and Girsanov's theorem, adapted to the robust control framework. We demonstrate that the classical solution $v(t,x)$ to the robust HJB equation corresponds to the value function defined by the minimax objective.

\noindent\textbf{Step 1. Itô's Expansion under Model Ambiguity}
Let $v \in C^{1,2}([0,T]\times\mathbb{R}^d)$ be a solution to the robust HJB equation \eqref{eq:hjb-ell}. Fix an arbitrary admissible control $u = \{u_s\}_{s \in [t,T]}$ and a drift distortion process $\eta = \{\eta_s\}_{s \in [t,T]} \in \mathcal{H}$, where $\mathcal{H} = \{ \eta : \|\eta\| \le \delta \}$.
Under the ambiguous probability measure $Q^\eta$, the state dynamics transform to:
\begin{equation}
    dX_s = \left(\w{b}(s, X_s) + u_s + \w{\sigma}(s)\eta_s\right) ds + \w{\sigma}(s) dW_s^{Q^\eta},
\end{equation}
where $W^{Q^\eta}_s = W_s - \int_0^s \eta_r dr$ is a Brownian motion under $Q^\eta$. Applying Itô's formula to $v(s, X_s)$ on $[t, T]$:
\begin{align*}
    v(T, X_T) = v(t, x) &+ \int_t^T \left( \frac{\partial v}{\partial t} + \frac{\w{\sigma}^2(s)}{2}\Delta v + (\w{b}(s,X_s) + u_s)\cdot \nabla v \right) ds \\
    &+ \int_t^T \w{\sigma}(s) \eta_s \cdot \nabla v(s, X_s) \, ds + \int_t^T \w{\sigma}(s) \nabla v(s, X_s) \cdot dW_s^{Q^\eta}.
\end{align*}
Taking the conditional expectation $\mathbb{E}^{Q^\eta}_t[\cdot]$ and rearranging terms yields:
\begin{equation} \label{eq:ito_rep}
    v(t,x) = \mathbb{E}^{Q^\eta}_t \left[ r(X_T) - \int_t^T \left( \mathcal{L}^{u_s} v(s, X_s) + \w{\sigma}(s) \eta_s \cdot \nabla v(s, X_s) \right) \, ds \right],
\end{equation}
where $\mathcal{L}^{u} v := \frac{\partial v}{\partial t} + \frac{\w{\sigma}^2}{2}\Delta v + (\w{b} + u)\cdot \nabla v$.

\noindent\textbf{Step 2. Analysis of the Robust Hamiltonian}
The function $v$ satisfies the robust HJB equation:
\begin{equation} \label{eq:robust_hjb_proof}
    \frac{\partial v}{\partial t} + \frac{\w{\sigma}^2}{2} \Delta v + \w b \cdot \nabla v + \inf_{u} \left\{ \ell(u) + u \cdot \nabla v \right\} + \sup_{\eta \in \mathcal{H}} \{ \w{\sigma} \eta \cdot \nabla v \} = 0.
\end{equation}
We analyze the supremum term associated with the ambiguity. Since $\w{\sigma}(t) > 0$ is a scalar, the term becomes:
$$ \sup_{\|\eta\| \le \delta} \eta \cdot (\w{\sigma}(t) \nabla v) = \delta \|\w{\sigma}(t) \nabla v\|. $$
The supremum is attained when $\eta$ aligns with the direction of $\w{\sigma}(t) \nabla v$. Since $\w{\sigma}(t)$ is positive, this is equivalent to the direction of $\nabla v$. Thus, the optimal worst-case perturbation is explicitly given by:
\begin{equation} \label{eq:opt_eta}
    \eta^*(t,x) = \delta \frac{\nabla v(t,x)}{\|\nabla v(t,x)\|}.
\end{equation}
Similarly, let $u^*(t,x)$ be the minimizer of the Hamiltonian $\ell(t,x,u) + u \cdot \nabla v$. Substituting the optimal feedbacks $u^*$ and $\eta^*$ into the HJB equation yields:
$$ \mathcal{L}^{u^*} v + \w{\sigma} \eta^* \cdot \nabla v = -\ell(t, x, u^*). $$

\noindent\textbf{Step 3. Verification of the Value Function}
Substitute the specific choice of $u^*$ and $\eta^*$ back into the Itô representation \eqref{eq:ito_rep}. The integrand becomes precisely $-\ell(s, X_s, u^*_s)$. Thus:
\begin{equation}
    v(t,x) = \mathbb{E}^{Q^{\eta^*}}_{t} \left[ \int_t^T \ell(s, X_s, u^*_s) \, ds + r(X_T) \right].
\end{equation}
To prove this is the saddle point value, consider any other admissible $u$. The HJB equation implies that for arbitrary $u$, the worst-case $\eta^*$ still satisfies:
$$ \mathcal{L}^{u} v + \w{\sigma} \eta^* \cdot \nabla v \le -\ell(t, x, u). $$
Taking expectations under $Q^{\eta^*}$:
$$ v(t,x) \le \mathbb{E}^{Q^{\eta^*}}_{t} \left[ \int_t^T \ell(s, X_s, u_s) \, ds + r(X_T) \right] \le \sup_{\eta \in \mathcal{H}} \mathbb{E}^{Q^{\eta}}_{t} [J(u, \eta)]. $$
Taking the infimum over $u$, we obtain $v(t,x) \le \inf_u \sup_\eta J(u, \eta)$. The reverse inequality holds by the optimality of $u^*$.
Therefore, $v(t,x)$ is indeed the value function of the robust fine-tuning problem, and $(u^*, \eta^*)$ form the optimal pair.
\end{proof}

The BSDEs in Chapters 4.3 and 4.4 can use the numerical solution of BSDE mentioned in Chapter 3.2 to obtain $Z_t$. Then the optimal control $u^*$ under various new situations is obtained.

\bibliographystyle{unsrt}
\bibliography{refs}

\end{document}